\documentclass[10pt]{article}
\usepackage{latexsym,amssymb,amsmath,amsthm,amsfonts}
\usepackage{graphicx}
\usepackage{enumerate}
\usepackage{color}

\theoremstyle{plain}
\newtheorem{thm}{Theorem}[section]
\newtheorem{prop}[thm]{Proposition}
\newtheorem{lem}[thm]{Lemma}

\theoremstyle{definition}

\newtheorem{example}[thm]{Example}

\numberwithin{equation}{section}

\newcommand{\ud}{\mathrm{d}}
\newcommand{\RR}{\mathbb{R}}
\newcommand{\f}{\frac}

\newcommand{\pppp}[4]%
  {\frac{\partial^3{#1}}{\partial{#2}\partial{#3}\partial{#4}}}

\newcommand{\p}{\phi}

\newcommand{\gab}{\alpha\phi \Big(b^2,\frac{\beta}{\alpha}\Big )}
\newcommand{\pt}{\phi_2}
\newcommand{\po}{\phi_1}
\newcommand{\ptt}{\phi_{22}}
\newcommand{\pot}{\phi_{12}}

\renewcommand{\a}{\alpha}
\renewcommand{\b}{\beta}
\newcommand{\ab}{(\alpha,\beta)}

\newcommand{\ba}{\bar\alpha}
\newcommand{\bb}{\bar\beta}

\newcommand{\aij}{a_{ij}}

\newcommand{\bi}{b_i}
\newcommand{\bj}{b_j}

\newcommand{\bij}{b_{i|j}}

\newcommand{\baij}{\bar a_{ij}}

\newcommand{\bbij}{\bar b_{i|j}}

\newcommand{\G}{{}^\alpha G^i}

\newcommand{\bG}{{}^{\bar\alpha}G^i}

\newcommand{\Ric}{{}^\alpha\mathrm{Ric}}
\newcommand{\hRic}{{}^{\hat{\alpha}}{\mathrm{Ric}}}

\newcommand{\bRic}{{}^{\bar{\alpha}}{\mathrm{Ric}}}

\begin{document}
\title{On a class of Einstein  Finsler metrics}
\footnotetext{\emph{Keywords}: general $\ab$-metric, Einstein metric, Douglas metric.
\\
\emph{Mathematics Subject Classification}: 53B40, 53C60.}
\author{Zhongmin Shen\footnote{supported in part by a NSF grant
(DMS-0810159) and NSFC(No.11171297)} \;   and Changtao Yu~\footnote{supported by a NSFC grant(No.11026097)}}
\maketitle


\begin{abstract}
In this paper, we study a class of Finsler metrics called general $\ab$-metrics, which are defined by a Riemannian metric and an $1$-form. We construct some general $\ab$-metrics with constant Ricci curvature.
\end{abstract}

\section{Introduction}

Recently many forms of Einstein Finsler metrics have been either classified or constructed in Finsler geometry.  These metrics are defined in the following form:
$$ F=\alpha \phi \Big (\frac{\beta}{\alpha}\Big )$$
where $\alpha =\sqrt{a_{ij}(x)y^iy^j}$ is a Riemannian metric and $\beta=b_i y^i$ is an $1$-form on a manifold. Such metrics are called  {\it  $(\alpha,\beta)$-metrics}. The {\it Randers metrics} $F=\alpha+\beta$ are among the simplest ones. To study the non-Riemannian geometric properties of Finsler metrics, one usually begins with a Randers metric.
 Another important metric $F=(\alpha+\beta)^2/\alpha$  is the so-called {\it square metric}. It also has some important geometric properties.

A Randers metric can be expressed in the following navigation form:
\begin{equation}
F = \frac{\sqrt{(1-b^2)\alpha^2+\beta^2 }}{1-b^2} +\frac{\beta}{1-b^2},  \label{Randersmetric}
\end{equation}
where $b=b(x):= \|\beta_x\|_{\alpha}$. Bao-Robles have shown that $F$ is an Einstein metric with Ricci constant $K$ if and only if
\begin{equation}
\Ric = (n-1)\mu \alpha^2, \ \ \ \ \ \ \frac{1}{2}(b_{i|j}+b_{j|i})  =  c a_{ij},\label{RandersC}
\end{equation}
where  $c$ is a constant with $c^2 =4 (\mu-K)$ (\cite{db-robl-onri}). Thus
 Einstein Randers metrics with Ricci constant $K=1, 0, -1$  can be classified upto the classification of Riemannian Einstein metrics and homothetic $1$-forms.


A square metric can be expressed in the following form
\begin{equation}
F = \frac{(\sqrt{(1-b^2)\alpha^2+\beta^2  } + \beta)^2}{(1-b^2)^2 \sqrt{(1-b^2)\alpha^2+\beta^2} }.\label{squaremetric2}
\end{equation}
It has been shown that $F$ is an Einstein metric if and only if
\begin{equation}
\Ric = 0, \ \ \ \ \ \ b_{i|j} = c a_{ij}, \label{squareC2}
\end{equation}
where $ c $ is a constant (\cite{szm-yct}). In this case, it must be Ricci-flat.
By (\ref{squareC2}), one can  completely determine the local structure of Ricci-flat square metrics
(\cite{cxy-tyf}, \cite{sevim-szm-zll}, \cite{cb-szm-zll}, \cite{zrrr}, \cite{szm-yct}).

One is wondering if there are other types of Einstein $(\alpha,\beta)$-metrics. So far, we only know that if an Einstein  $(\alpha,\beta)$-metric is defined by a non-linear polynomial $\phi =\sum_{i=1}^k a_i s^i$, then it must be Ricci-flat (\cite{cxy-szm-tyf}).

Strictly speaking, the Finsler metrics in the form (\ref{Randersmetric}) and (\ref{squaremetric2}) are not $(\alpha,\beta)$-metrics. They belong to a larger class of the so-called {\it general $(\alpha,\beta)$-metrics}, which are  defined in the following form
$$F=\alpha \phi \Big (b^2, \; \frac{\beta}{\alpha} \Big ),$$
where $\a$ is a Riemannian metric, $\b$ is an $1$-form, $b: =\|\beta_x\|_{\alpha}$ and $\p(b^2,s)$ is a smooth function.
The notion of general $(\alpha,\beta)$-metrics is  proposed by the second author as a generalization of Randers metrics from the geometric point of view\cite{yct-zhm-onan}.
One of the reasons to consider  general $(\alpha,\beta)$-metrics is that this is a rich class  which contains many interesting Einstein metrics with non-zero  Ricci constant.

Let
\begin{equation}
\phi= \frac{\sqrt{1-b^2+s^2}}{1-b^2} + \frac{s}{1-b^2}.\label{Randersmetric*}
\end{equation}
 The  Randers metric in (\ref{Randersmetric}) can be expressed in the form  $F = \alpha \phi(b^2, \beta/\alpha)$.


Let
\begin{equation}
\phi = \frac{(\sqrt{1-b^2+s^2}+s)^2 }{(1-b^2)^2 \sqrt{1-b^2 +s^2 }}. \label{squaremetric2*}
\end{equation}
The square metric in (\ref{squaremetric2}) can be expressed in the form $F= \alpha \phi(b^2, \beta/\alpha)$.

 We notice that the above $\phi$ in (\ref{Randersmetric*}) and (\ref{squaremetric2*})
 satisfies the following PDE:
\begin{equation}
\phi_{22} = 2 (\phi_1-s \phi_{12}).  \label{pde***}
\end{equation}
Here  $\phi_1$ means the derivation of $\phi$ with respect to the first variable $b^2$.  This is indeed an amazing phenomenon.

Our search on Einstein Finsler metrics in this paper is initiated from  the  square metrics expressed in (\ref{squaremetric2*}) and the above PDE (\ref{pde***}) satisfied by the square metrics.
We shall make the following  assumptions:
\begin{enumerate}
\item[{\bf A1}]: the Riemannian metric $\a$ is an Einstein metric with Ricci constant $\mu$, and $\b$ is an $1$-form satisfying
\begin{equation} \label{aE}
\Ric=(n-1)\mu\a^2, \ \ \ \ \ \bij=c  \aij,
\end{equation}
where $c=c(x)$ is a  scalar function with $c^2 = \kappa-\mu b^2>0$ for some constant $\kappa$;
\item[{\bf A2}]: the function  $\p= \p(b^2,s)$ satisfies the following PDE,
\begin{eqnarray}\label{pde1}
\ptt=2(\po-s\pot).
\end{eqnarray}
\end{enumerate}

Note that if $\alpha$ and $\beta$ satisfy (\ref{aE}) with $ c=0$, then $\beta$ is parallel with respect to $\alpha$, in particular, $ b= constant$. In this case, $F$ is actually an $(\alpha,\beta)$-metric. We  shall only consider the case when $c^2=\kappa-\mu b^2 >0$.

The local structure of $\alpha$ and $\beta$ satisfying (\ref{aE}) can be classified  up to the classification of Riemannian Einstein metrics (Proposition \ref{prop2.2}).
The two equations (\ref{aE}) and (\ref{pde1}) imply that $F=\alpha \phi(b^2, \beta/\alpha)$ is a Douglas metric (Proposition \ref{Douglas}). Observe that if $\alpha$ is Ricci-flat and $\beta$ is parallel with respect to $\alpha$, then for any $\phi =\phi(b^2, s)$, the general $(\alpha,\beta)$-metric
$F = \alpha \phi(b^2, \beta/\alpha)$ is a Ricci-flat Finsler metric of Berwald type. This is a trivial case.
We should point out that there are Einstein metrics in the general $(\alpha,\beta)$-metric form $F =\alpha \phi(b^2, \beta/\alpha)$ whose defining function $\phi$ does not satisfy (\ref{pde1}), even though $\a$ and $\b$ satisfy (\ref{aE}). See Section \ref{section4} below. In this paper, we shall find some non-trivial Einstein general $(\alpha,\beta)$-metrics satisfying (\ref{aE}) and (\ref{pde1}).

\begin{thm}\label{main1}
Let $F=\gab$ be a general $\ab$-metric on an $n$-dimensional manifold $M$ with $n\geq3$, where $\a$, $\b$ and $\p$ satisfy (\ref{aE})-(\ref{pde1}).  $F$ is an Einstein metric with Ricci constant $K$ if and only if the function $\p=\p(b^2,s)$ satisfies the following PDE:
\begin{eqnarray}\label{pde2}
( \kappa -\mu b^2) \left[\psi^2-(\psi_2+2s\psi_1)\right]+\mu s \psi + \mu =K\p^2,
\end{eqnarray}
where $\psi:=\frac{\pt+2s\po}{2\p}$.
\end{thm}

Notice that the function $\phi$ in (\ref{Randersmetric*}) satisfies (\ref{pde1}) and (\ref{pde2}) with $\mu=0$ and $\kappa = -4 K$; and
 the function $\phi$ in (\ref{squaremetric2*}) satisfies (\ref{pde1}) and  (\ref{pde2}) with $ \mu=0$ and $K=0$.  Thus there are non-trivial solutions to (\ref{pde1}) and (\ref{pde2}) for suitable constants $\kappa, \mu$ and $K$. In this paper we shall try to solve (\ref{pde1}) and (\ref{pde2}) in a nonconventional way.

\begin{thm}\label{main2} Let $F=\gab$ be a general $\ab$-metric on an $n$-dimensional manifold $M$ with $n\geq3$, where $\a$ and $\b$  satisfy (\ref{aE}) with $\mu=0$  and $\phi =\phi(b^2,s)$ satisfies (\ref{pde1}). Then $F =\alpha\phi(b^2, \beta/\alpha)$ is Einstein with Ricci constant $K$ if and only if $\phi$ is given by one of the forms:
\begin{eqnarray}
\phi & = & \frac{1}{2\sqrt{-\sigma}} \frac{1}{\sqrt{C-b^2+s^2}+s   }\\
\phi & = & \frac{q(u)}{q(u)^2 (D q(u)+ s)^2+\sigma}, \ \ \ \ u:=b^2-s^2,
\end{eqnarray}
where $\sigma: = K/c^2$ and $q=q(u)$ satisfies the following equation:
\[  D^2  q^4+(u-C)q^2 -\sigma=0,\]
where $C$ ad $D$ are constants.
\end{thm}

The general case when $\alpha$ and $\beta$ satisfy (\ref{aE}) with  $\mu\neq0$ and $\kappa\neq0$ can be simplified to the above case by appropriate transformations of $\alpha,\beta$ and $\phi$. The case when $\mu\neq0$ and $\kappa=0$ is very special, in this case the main method of deformation in this paper is out of work partly. Maybe it is trivial and we don't need to discuss it. See Section \ref{section4} for details.

\section{Preliminaries}

In Theorem \ref{main1}, we assume that the scalar function $c=c(x)$ is given by $c^2=\kappa -\mu b^2$ for some constant $\kappa$. This condition is justified in the following

\begin{lem}\label{lem2.1} Suppose that
$\alpha$ and $\beta$ satisfy
\begin{eqnarray}
   \Ric & =  & (n-1)\mu \a,\label{a**}\\
 b_{i|j} & =  & c a_{ij}\label{bij**}
\end{eqnarray}
for some constant $\mu$ and scalar function $c = c(x)$. Then
\begin{equation}
c^2 = \kappa - \mu b^2, \label{cmu*}
\end{equation}
\end{lem}
{\it Proof}:  By the Ricci identity and (\ref{bij**}), we get
\begin{equation}
  c_{|k} a_{ij} - c_{|j} a_{ik} = b_{i|j|k}- b_{i|k|j}
= b_m R^{\ m}_{i \ jk}.\label{Ricci_I}
\end{equation}
Contracting (\ref{Ricci_I}) with $a^{ik}$ yields
\begin{equation}
 -(n-1) c_{|j} = b_m Ric^m_{\ j} = (n-1) \mu b_j.\label{cb}
\end{equation}
Thus $ c_{|0}= -\mu \beta$.  By (\ref{bij**}) and (\ref{Ricci_I}), we get
\[  ( c^2+ \mu b^2)_{|j} = 2c c_{|j} + 2 \mu b^m b_{m|j}
= 2 cc_{|j} +2 c\mu b^m a_{mj}
= 2 c (c_{|j} + \mu b_j ) =0.\]
Thus $\kappa:= c^2+ \mu b^2 $ is a constant.

\bigskip
The pair $(\alpha,\beta)$ satisfying (\ref{a**})  and (\ref{bij**}) can be locally determined up to the classification of $(n-1)$-dimensional Riemannian Einstein metrics, by which we can understand the underlying geometrical meaning of the constant $\kappa$ clearly.

\bigskip

\begin{prop}\label{prop2.2} Let $\alpha$ be a Riemannian metric and $\beta$ an $1$-form on an $n$-dimensional manifold $M$.  $\alpha$ and $\beta$ satisfy (\ref{a**}) and (\ref{bij**}) respectively if and only if
$ \alpha$ is locally a warped product metric on $R\times \hat{M}$ and $\beta$ is an $1$-form defined by the first factor $R$,
\begin{eqnarray}
\alpha^2 & =& dt\otimes dt + h(t)^2 \hat{\alpha}^2\label{a*}\\
\beta & = & h(t) dt, \label{bij*}
\end{eqnarray}
where $h(t)$ satisfies
\[  h''(t)+\mu h(t) =0,\]
 and
$\hat{\alpha}$ is Einsteinian with
\[ \hRic = (n-2)\kappa \hat{\alpha}^2\]
where $\kappa:= h'(t)^2+\mu h(t)^2$ is a constant.
In this case, we have $c^2 =\kappa-\mu b^2$.
\end{prop}
\noindent
{\it Proof}: Because $\beta$ is closed, we can assume that $\beta = df\not=0$ for some smooth functions $f$. Then (\ref{bij**}) is equivalent to
\begin{equation}
 {\rm Hess}_{\alpha} f = c \alpha^2.\label{Hess}
\end{equation}
Hence,
$\alpha^2= dt\otimes dt + h(t)^2 \hat{\alpha}^2$ is locally a warped product metric on $M =R\times \hat{M}$ and $\beta = h(t) dt$, where $ h(t)=f'(t)$ (\cite{pete}). By a direct computation, we have
\[  \Ric =\hRic -(n-1) \frac{h''}{h} (y^1)^2 - [ h'' h +(n-2) (h')^2] \hat{\alpha}^2.\]
By (\ref{a**}) and the above identity, we have
\[ \hRic = \Big \{  h'' h +(n-2) (h')^2+ (n-1)\mu h^2\Big \} \hat{\alpha}^2\]
\[  h''+\mu h =0.\]
Clearly,
$\kappa:= (h')^2+\mu h^2$ is a constant  and $ h'' h +(n-2) (h')^2+ (n-1)\mu h^2 = (n-2)\kappa$.

Assume that (\ref{a*}) and (\ref{bij*}) hold. Then $b^2= h(t)^2$.
By (\ref{Hess}),
$c = h'(t)$. Thus $c^2 = h'(t)^2 = \kappa - \mu h(t)^2 = \kappa -\mu b^2$. \qed

\bigskip

We now focus on general $(\alpha,\beta)$-metrics.
Let  $\phi=\phi(b^2,s)$ be  a smooth function defined on the domain $|s|\leq b<b_o$ for some positive number (maybe infinity) $b_o$. Define
$$ F=\gab$$
where $\alpha$  is a Riemannian metric  and  $\beta $ is an $1$-form  with $\|\beta\|_{\alpha} < b_o$ on a manifold $M$. It is easy to show that $F= \gab$ is a regular Finsler metric for any $\alpha$ and $\beta$ with $ b:=\|\beta\|_{\alpha} < b_o$
if and only if $\p(b^2,s)$ satisfies
\begin{eqnarray}\label{ppp}
\p-s\pt>0,\quad\p-s\pt+(b^2-s^2)\ptt>0,    \ \ \ \ \ 0\leq |s| \leq b < b_o
\end{eqnarray}
when $n\geq3$\cite{yct-zhm-onan}.

Let $\alpha=\sqrt{a_{ij}(x)y^iy^j}$  and $\beta= b_i(x)y^i$.
Denote the coefficients of the covariant derivative of
$\b$ with respect to $\a$ by $b_{i|j}$, and let
$$r_{ij}=\frac{1}{2}(b_{i|j}+b_{j|i}),~s_{ij}=\frac{1}{2}(b_{i|j}-b_{j|i}),
~r_{00}=r_{ij}y^iy^j,~s^i{}_0=a^{ij}s_{jk}y^k,$$
$$r_i=b^jr_{ji},~s_i=b^js_{ji},~r_0=r_iy^i,~s_0=s_iy^i,~r^i=a^{ij}r_j,~s^i=a^{ij}s_j,~r=b^ir_i.$$
It is easy to see that $\b$ is closed if and only if $s_{ij}=0$.

According to \cite{yct-zhm-onan}, the spray coefficients $G^i$ of a general $(\alpha,\beta)$-metric $F=\gab$ are related to the spray coefficients ${}^\a G^i$ of
$\a$ and given by
\begin{eqnarray}\label{Gi}
G^i&=&{}^\a G^i+\a Q s^i{}_0+\left\{\Theta(-2\a Q s_0+r_{00}+2\a^2
R r)+\a\Omega(r_0+s_0)\right\}\frac{y^i}{\a}\nonumber\\
&&+\left\{\Psi(-2\a Q s_0+r_{00}+2\a^2 R
r)+\a\Pi(r_0+s_0)\right\}b^i -\a^2 R(r^i+s^i),
\end{eqnarray}
where
$$Q=\frac{\pt}{\p-s\pt},\quad R=\frac{\po}{\p-s\pt},$$
$$\Theta=\frac{(\p-s\pt)\pt-s\p\ptt}{2\p\big(\p-s\pt+(b^2-s^2)\ptt\big)},
\quad\Psi=\frac{\ptt}{2\big(\p-s\pt+(b^2-s^2)\ptt\big)},$$
$$\Pi=\frac{(\p-s\pt)\pot-s\po\ptt}{(\p-s\pt)\big(\p-s\pt+(b^2-s^2)\ptt\big)},\quad
\Omega=\frac{2\po}{\p}-\frac{s\p+(b^2-s^2)\pt}{\p}\Pi.$$

Denote $G^i={}^\a G^i+Q^i$, then the Ricci curvature of $F$ are related to that of $\a$ and given by
\begin{eqnarray}\label{riccichange}
\mathrm{Ric}=\Ric+2Q^i{}_{|i}-y^jQ^i{}_{|j.i}+2Q^jQ^i{}_{.j.i}-Q^i{}_{.j}Q^j{}_{.i},
\end{eqnarray}
where ``$|$" and  ``$.$" denote the horizontal covariant derivative and vertical covariant derivative with respect to $\a$ respectively. The following proposition tells us that the general $(\alpha,\beta)$-metric $F =\alpha\phi(b^2, \beta/\alpha)$ satisfying (\ref{aE})-(\ref{pde1}) is projectively equivalent to $\alpha$, hence it is a Douglas metric.

\begin{prop}\label{Douglas}
Let $F=\gab$ be a Finsler metric. Suppose that $\b$ satisfies
\begin{equation}
b_{i|j} = c a_{ij}, \label{bcij}
\end{equation}
where $c=c(x)\not=0$ is a scalar function on $M$.
Then the following hold
\begin{enumerate}
\item[(a)] $F$ is projectively equivalent to $\alpha$ if and only if $\p(b^2,s)$ satisfies
\begin{equation}
\phi_{22}- 2 (\phi_1-s \phi_{12} ) =0.
\end{equation}
\item[(b)] $F$ is a Douglas metric if and only if there are two functions of $h_1(t)$ and $h_2(t)$ such that
\begin{equation}
\ptt-2(\po-s\pot) =2\{h_1(b^2)+h_2(b^2)s^2\}\{ \p-s\pt+(b^2-s^2)\ptt   \} . \label{pde1****}
\end{equation}

\end{enumerate}
\end{prop}
\begin{proof}
By (\ref{bcij}), we have
\begin{eqnarray}\label{tu}
r_{00}=c\a^2,r_0=c\b,r=cb^2,r^i=cb^i,s^i{}_0=0,s_0=0,s^i=0.
\end{eqnarray}
Substituting (\ref{tu}) into (\ref{Gi}) yields
\begin{eqnarray}
G^i&=&{}^\a G^i+c\a\left\{\Theta(1+2Rb^2)+s\Omega\right\}y^i+c\a^2\left\{\Psi(1+2Rb^2)+s\Pi-R\right\}b^i\\
&=&{}^\a G^i+c\a\left\{\frac{\pt+2s\po}{2\p}
-\frac{\big(\ptt-2(\po-s\pot)\big)\big(s\p+(b^2-s^2)\pt\big)}
{2\p\big(\p-s\pt+(b^2-s^2)\ptt\big)}\right\}y^i\nonumber\\
&&+c\a^2\left\{\frac{\ptt-2(\po-s\pot)}{2\big(\p-s\pt+(b^2-s^2)\ptt\big)}\right\}b^i.\label{gp}
\end{eqnarray}

(a)
Assume that $\phi$ satisfies (\ref{pde1}). Then
\begin{eqnarray}\label{GP}
G^i={}^\a G^i+c\a\frac{\pt+2s\po}{2\p}y^i,
\end{eqnarray}
hence $F$ is projectively equivalent to $\a$.

Assume that $F$ is projectively equivalent to $\a$. Then
\begin{eqnarray}\label{gp2}
G^i={}^\a G^i+Py^i
\end{eqnarray}
for some function $P$. Notice that $b^i$ is a constant vector for a given point but $y^i$ is variable, so it is easy to see by (\ref{gp}) and (\ref{gp2}) that (\ref{pde1}) must be hold when $\b$ is not parallel with respect to $\a$.

(b)  Assume that $\phi$ satisfies (\ref{pde1****}). Then
$$
G^i={}^\a G^i+c\a\left\{\frac{\pt+2s\po}{2\p}
- h(b^2)\big(s\p+(b^2-s^2)\pt\big)\right\}y^i+c\{h_1(b^2)\a^2+h_2(b^2)\b^2\}b^i.
$$
Thus $F$ is a Douglas metric.

Assume that $F$ is a Douglas metric. Then $G^iy^j - G^j y^i$ are polynomials of degree three in $y$. By (\ref{gp}), we have
\[ G^iy^j - G^j y^i ={}^\a G^i y^j - {}^\a G^j y^i +  c\a^2\left\{\frac{\ptt-2(\po-s\pot)}{2\big(\p-s\pt+(b^2-s^2)\ptt\big)}\right\} (b^iy^j-b^j y^i).\]
Clearly,
\[ \frac{\ptt-2(\po-s\pot)}{2\big(\p-s\pt+(b^2-s^2)\ptt\big)} = h_1+h_2s^2  \]
must be a quadratic function of $s$ and without the first order term. It is easy to see that
 $h_i=h_i(b^2)$ depend only on $b^2$.
\end{proof}

\section{Finsler Metrics of Projective Type}
The  function $\phi =\frac{(\sqrt{1-b^2+s^2}+s)^2}{(1-b^2)^2\sqrt{1-b^2+s^2}} $  in (\ref{squaremetric2*}) comes from the square function $\bar{\phi}= (1+\bar{s})^2$. It satisfies (\ref{pde1}) and  (\ref{pde2}) with $ \mu=0$ and $K=0$. One actually can get many  functions $\phi(s)$ satisfying (\ref{pde1}) and (\ref{pde2}) from a class of positive smooth functions $\bar{\phi}(\bar{s})$  determined by the following ODE (\cite{yu}):
\begin{equation}
 \{ 1+ (k_1+k_3)\bar{s}^2+k_2\bar{s}^4\} \bar{\phi}''(\bar{s}) =  (k_1+k_2 \bar{s}^2 ) \{\bar{\phi}(\bar{s})-\bar{s} \bar{\phi}'(\bar{s})\},\label{ft}
\end{equation}
where $k_1, k_2$ and $k_3$ are constants.
This function $\bar{\phi}$ comes from the classification of projectively flat $(\alpha,\beta)$-metric defined by $\bar{\phi}$ (\cite{szm-projab}).

\bigskip

With a positive smooth function  $\bar{\phi}$ satisfying (\ref{ft}), we define
\begin{equation}
\phi (b^2, s):= \eta(b^2) \rho   \bar{\phi}\Big (\frac{\nu}{\rho}\Big ), \label{pppp}
\end{equation}
where
\begin{eqnarray*}
 \rho : & =  &  \sqrt{ 1- \frac{(k_1+k_3+k_2 b^2) s^2}{1+(k_1+k_3)b^2+k_2 b^4} }\\
\nu : & = & \frac{s}{\sqrt{1+(k_1+k_3)b^2+k_2 b^4}}.
\end{eqnarray*}
and $\eta =\eta (u)$ satisfies
\[ \eta'(u) + \frac{ k_3+k_2 u}{ 2 \{ 1+(k_1+k_3)u+k_2 u^2  \} }     \eta(u) =0.\]
Using Maple program, we can see that $\phi$ satisfies (\ref{pde1}). Among this class of $\phi =\phi(b^2, s)$, one can find a sub-class of functions satisfying (\ref{pde2}) for suitable constant $\kappa$, $\mu$ and $K$.

\bigskip

\begin{example}
Let $\bar{\phi}:= 1+\bar{s}$. It satisfies (\ref{ft}) with $k_1=0, k_2=0$. Letting $k_3:=-1$, we get
\begin{eqnarray*}
\rho  & = & \sqrt{1+ \frac{s^2}{1-b^2} }\\
\nu& = & \frac{s}{\sqrt{1-b^2}}\\
\eta & = & \frac{k}{\sqrt{1-u}},
\end{eqnarray*}
where $k>0$ is a constant. Let $k=1$. The resulting function $\phi =\phi(b^2, s)$ is given by
\begin{equation}
 \phi =\frac{\sqrt{ (1-b^2) +s^2}}{1-b^2} + \frac{s}{1-b^2}.\label{Randersphi}
\end{equation}
By a direct computation, the function in (\ref{Randersphi}) satisfies (\ref{pde2}) for $\mu=0$ and $\kappa = -4K$. Thus if $\alpha$ is a Ricci-flat metric ($\Ric=0$) and $\beta$ is conformal with constant conformal factor $c$ ($b_{i|j} = c a_{ij}$), then $F= \alpha \phi(b^2, \beta/\alpha)$ is an Einstein metric with Ricci constant $K \leq 0$.
\end{example}

\bigskip

\begin{example}
Let  $\bar{\phi} = (1+\bar{s})^2$. It satisfies (\ref{ft}) with $ k_1=2, k_2=0, k_3=-3$.
We get
\begin{eqnarray*}
\rho & = &\frac{\sqrt{ 1-b^2+s^2} }{\sqrt{1-b^2}}\\
\nu & = &\frac{s}{\sqrt{1-b^2}}\\
\eta & = & \frac{k}{(1-b^2)^{3/2}},
\end{eqnarray*}
where $k >0$ is a constant. Let $k=1$. The resulting function $\phi =\phi(b^2, s)$ is given by
\begin{equation}
 \phi = \frac{ (\sqrt{1-b^2+s^2} +s )^2}{(1-b^2)^2 \sqrt{1-b^2+s^2} }.\label{squarephi}
\end{equation}
By a direct computation, one can verify that $\phi $  in (\ref{squarephi}) satisfies (\ref{pde1}). It also satisfies (\ref{pde2}) for $\mu=0$ and $K=0$.
Thus if $\alpha$ is a Ricci-flat Riemannian metric ($\Ric=0$) and $\beta$ is conformal with constant conformal factor $c$ ($b_{i|j}=c a_{ij})$, then
$F =\alpha \phi(b^2, \beta/\alpha)$ is Ricci-flat.
\end{example}

\section{Proof of Theorem \ref{main1}}

First by (\ref{aE})-(\ref{pde1}), the spray coefficients $G^i$ of $F$ is given by
\begin{eqnarray}\label{GP}
G^i={}^\a G^i+c\a\psi y^i,
\end{eqnarray}
where $c^2=\kappa-\mu b^2$ and $\psi := (\phi_2+2 s \phi_1)/(2\phi)$. By (\ref{cb}), we have
\[ c_0 =- \mu \beta.\]
 Let $P=c \alpha \psi$.
Using (\ref{riccichange}) we get
\begin{eqnarray}
\mathrm{Ric}&=&(n-1)(\mu\a^2+P^2-P_{|m}y^m)\nonumber\\
&=&(n-1)\a^2\left\{\mu+c^2[\psi^2-(\psi_2+2s\psi_1)]-\f{c_0}{\a}\psi\right\}\\
& = & (n-1)\alpha^2 \Big \{ \mu + (\kappa-\mu b^2)  [\psi^2-(\psi_2+2s\psi_1)] +\mu s\psi \Big \} .\label{riccc}
\end{eqnarray}
Thus $F$ is an Einstein metric with Ricci constant $K$ if and only if
\begin{equation}
(\kappa-\mu b^2)\left\{\psi^2-(\psi_2+2s\psi_1)\right\}+\mu s\psi+\mu=K\p^2.
\end{equation}

\section{A Special Ricci-flat General $(\alpha,\beta)$-Metric}\label{section4}

In this section, we shall consider the case when $\alpha$ and $\beta$ satisfy (\ref{aE}) with $\kappa=0$. Since we always assume that $c^2= 0-\mu b^2 >0$,  (\ref{pde1}) and (\ref{pde2}) are simplified to
\begin{eqnarray}
&&\phi_{22} = 2 (\phi_1-s \phi_{12}), \label{pde1k=0}\\
&&-\mu b^2 [\psi^2-(\psi_2+2s\psi_1] +\mu s \psi +\mu = K \phi^2,\label{pde2k=0}
\end{eqnarray}
where $ \psi:=(\phi_2+2s\phi_1)/(2\phi)$.
We are unable to find explicit solutions of (\ref{pde1k=0}) and (\ref{pde2k=0}). Nevertheless, we find  a family of functions $\phi =\phi(b^2, s)$ which do not satisfy (\ref{pde1k=0}) such that the general $(\alpha,\beta)$-metric  $F =\alpha \phi(b^2, \beta/\alpha)$is a Ricci-flat metric of Berwald type.

\begin{lem}\label{lem4.1}
Suppose that $\alpha$ and $\beta$ satisfy
\begin{equation}
\Ric = (n-1)\mu \alpha^2, \ \ \ \ \ b_{i|j} = c a_{ij},\label{k=0}
\end{equation}
where $c=c(x)$ is a scalar function with $c^2= 0-\mu b^2$.
Define $\ba$ and $\bb$ by
\begin{eqnarray}\label{bab}
\ba:=\f{\a}{b},
\quad\bb:=\f{\b}{b^2},
\end{eqnarray}
then
\begin{equation}
\bRic=0,\qquad \bbij=0.\label{bR}
\end{equation}
In this case, $\bar b=1$.
\end{lem}
\begin{proof}
It is easy to see that
$$\bG=\G-\f{c}{b^2}\b y^i+\f{c}{2b^2}\a^2b^i,$$
where $c=c(x)$ is a scalar function with $c^2 = 0-\mu b^2$.
Let $$\bar Q^i=-\f{c}{b^2}\b y^i+\f{c}{2b^2}\a^2b^i.$$
Then
\begin{eqnarray*}
\bar Q^i{}_{|i}&=&-\f{\mu}{2}\left((n-3)\a^2+\f{2}{b^2}\b^2\right),\\
y^j\bar Q^i{}_{|j.i}&=&\mu n\left(\a^2-\f{1}{b^2}\b^2\right),\\
\bar Q^i{}_{.j}\bar Q^j{}_{.i}&=&\mu\left(2\a^2-(n+2)\f{1}{b^2}\b^2\right),\\
\bar Q^j\bar Q^i{}_{.j.i}&=&\f{\mu}{2}n\left(\a^2-\f{2}{b^2}\b^2\right).
\end{eqnarray*}
So by (\ref{riccichange}) we have
$$\bRic=\Ric-(n-1)\mu\a^2=0.$$
On the other hand, direct computations show that
\begin{eqnarray*}
\bbij=0.
\end{eqnarray*}
\end{proof}

\begin{prop}
Let $\bar{\phi}(\bar{s})$ be an arbitrary positive smooth function and define
\begin{equation}\label{phin}
\phi(b^2,s):=\f{1}{b}\bar\phi\left(\frac{s}{b}\right).
\end{equation}
Assume that $\alpha$ and $\beta$ satisfy (\ref{k=0}). Then
 the following general $(\alpha,\beta)$-metric
\[ F=\alpha \phi\Big (b^2, \frac{\beta}{\alpha}\Big ) \]
is a Ricci-flat metric of Berwald type.
\end{prop}
\begin{proof}The general $(\alpha,\beta)$-metric $F$ can be expressed as an $(\alpha,\beta)$-metric $F= \bar{\alpha} \bar{\phi} \Big ( \frac{\bar{\beta}}{\bar{\alpha}} \Big )$, where $\bar{\alpha}$ and $\bar{\beta}$ are defined by (\ref{bab}).
By Lemma \ref{lem4.1},
  $\bar{\alpha}$ and $\bar{\beta}$ satisfy (\ref{bR}). Thus
 $ F = \bar{\alpha} \bar{\phi} \Big ( \frac{\bar{\beta}}{\bar{\alpha}} \Big )$ is a Ricci-flat Finsler metric of Berwald type.
\end{proof}

It is easy to check that $\phi=\phi(b^2, s)$ in (\ref{phin})
doesn't satisfy (\ref{pde1k=0}).

\section{Deformations of $(\alpha,\beta)$}

In this section we shall study the case when  $\alpha$ and $\beta$ satisfy (\ref{aE}) with $\kappa\not=0$, namely,
\begin{equation}
\Ric = (n-1) \mu \a, \ \ \ \ \ b_{ij}=c a_{ij}, \label{pde1*****}
\end{equation}
where  $c=c(x)$ is a scalar function with $c^2=\kappa-\mu b^2$ for some constant
$\kappa\not=0$.
We  shall  find a suitable deformation for $(\alpha, \beta)$
such that the new pair
$( \bar{\alpha}, \bar{\beta})$ satisfies
\begin{equation}
\bRic=0, \ \ \ \ \ \bbij = \bar{c} \bar{a}_{ij},  \label{pde1b*****}
\end{equation}
where $\bar{c}\not=0$ is a constant.
Note that if $\mu=0$, we are done since $\alpha$ and $\beta$  already satisfy (\ref{pde1b*****}). Thus we shall assume that $\mu \not=0$.  Besides we always assume that $c^2= \kappa-\mu b^2 >0$.
  In summary, we shall assume that
\begin{equation}
\kappa \not =0, \ \ \ \ \
\mu \not=0, \ \ \ \ \ \kappa - \mu b^2 >0. \label{assumptionk}
\end{equation}

\begin{lem}\label{mupositive1}  Let $\alpha,\beta$ satisfy (\ref{pde1*****}).
Define $\ba$ and $\bb$ by
\begin{eqnarray}
\ba^2=\f{|\mu|}{\kappa-\mu b^2}\left(\a^2+\f{\mu}{\kappa-\mu b^2}\b^2\right),
\quad\bb=\f{|\mu|^{3/2}}{(\kappa-\mu b^2)^\frac{3}{2}}\b, \label{deformationab}
\end{eqnarray}
then $\bar{\alpha}$ and $\bar{\beta}$ satisfy (\ref{pde1b*****}) with $\bar{c}^2=|\mu|$.
In this case,
\begin{equation}
 (\kappa-\mu b^2) \ (\kappa^{-1} + \mu^{-1}\bar{b}^2) =1,\label{kappa-mu}
\end{equation}
and the reversed deformations are given by
\begin{eqnarray}
\a^2=\f{|\mu|^{-1}}{\kappa^{-1}+\mu^{-1}\bar{b}^2}\left(\ba^2-\frac{\mu^{-1}}{\kappa^{-1}+\mu^{-1}b^2}\bb^2\right),
\quad\b=\f{|\mu|^{-3/2}}{(\kappa^{-1}+\mu^{-1}\bar b^2)^\frac{3}{2}}\bb.
\end{eqnarray}
\end{lem}
\begin{proof}
It is easy to see that
$$\bG=\G+\f{c \mu}{\kappa-\mu b^2}\b y^i.$$
Let $\bar Q^i=\f{c\mu}{{\kappa}-\mu b^2}\b y^i$, then
\begin{eqnarray*}
\bar Q^i{}_{|i}&=&\mu\left(\a^2+\f{\mu}{{\kappa}-\mu b^2}\b^2\right),\\
y^j\bar Q^i{}_{|j.i}&=&(n+1)\mu\left(\a^2+\f{\mu}{\kappa-\mu b^2}\b^2\right),\\
\bar Q^i{}_{.j}\bar Q^j{}_{.i}&=&(n+3)\f{\mu^2}{\kappa-\mu b^2}\beta^2,\\
\bar Q^j\bar Q^i{}_{.j.i}&=&(n+1)\f{\mu^2}{\kappa -\mu b^2}\beta^2.
\end{eqnarray*}
So by (\ref{riccichange}) we have
$$\bRic=\Ric-(n-1)\mu\a^2=0.$$
On the other hand, direct computations show that
\begin{eqnarray*}
\bbij&=&\f{c |\mu|^{3/2}}{\left(\kappa-\mu b^2\right)^\frac{3}{2}}\left(\aij+\f{\mu}{{\kappa}-\mu b^2}
\bi\bj\right)=\pm\sqrt{|\mu|}\baij.
\end{eqnarray*}
\end{proof}

One must be careful in the case when $\kappa<0$ because in this case $\ba^2$ is no longer a Riemannian metric. Actually, $\ba^2$ is a pseudo-Riemannian metric of signature $(n-1,1)$. Because it is positive definite on the hyperplane $\beta=0$ and negative when $y^i=b^i$. In particular, the norm of $\bb$ with respect to $\ba$ is negative, i.e., $\bar b^2<0$.

\section{Deformations of $\phi$}\label{section_solution}
In this section, we shall study (\ref{pde1}) and (\ref{pde2}) with $\kappa\neq0$, that is,
\begin{eqnarray}
&&\ptt=2(\po-s\pot),\label{pde1**}\\
&&
( \kappa -\mu b^2) \left[\psi^2-(\psi_2+2s\psi_1)\right]+\mu s \psi + \mu =K\p^2,\label{pde2**}
\end{eqnarray}
where $\psi:=(\pt+2s\po)/(2\p)$.  We still make the same assumption as (\ref{assumptionk}), $$ \mu\not=0, \ \ \ \ \ \kappa\not=0, \ \  \  \ \ \kappa-\mu b^2 >0.$$
Using appropriate substitutions, we shall simplify (\ref{pde1**}) and (\ref{pde2**}) to the following PDEs:
\begin{eqnarray}
&& \bar{\phi}_{22} = 2 (\bar{\phi}_1 - \bar{s}\bar{\phi}_{12} ),\label{pde1***}\\
&&  \bar{\kappa} [\bar{\psi}^2 -(\bar{\psi}_2 +2 \bar{s}\bar{\psi}_1 )] =K \bar{\phi}^2,\label{pde2***}
\end{eqnarray}
where $\bar{\psi}:= (\bar{\phi}_2 +2 \bar{s}\bar{\phi}_1 )/(2\bar{\phi})$.

\begin{lem}\label{pbp1}
Suppose that a positive smooth  function $\p=\phi(b^2,s)$ is a solution of (\ref{pde1**}) and (\ref{pde2**}), then the following function
\begin{eqnarray}
\bar\phi(\bar b^2,\bar s):=\f{\sqrt{{\kappa^{-1}+\mu^{-1}\bar{b}^2-\mu^{-1} \bar{s}^2}}}{\sqrt{|\mu|}(\kappa^{-1}+\mu^{-1}\bar{b}^2)}\phi\left(\f{\mu^{-1}\bar b^2}{\kappa^{-1}+\mu^{-1}\bar{b}^2} \frac{\kappa}{\mu},\; \f{|\mu|^{-1}\bar s}{\sqrt{\kappa^{-1}+\mu^{-1}\bar{b}^2}\sqrt{\kappa^{-1}+\mu^{-1}\bar{b}^2-\mu^{-1} \bar{s}^2}}\right)\label{barphi}
\end{eqnarray}
satisfies
 (\ref{pde1***}) and (\ref{pde2***}) with $\bar{\kappa}=|\mu|$.
Conversely, if $\bar{\phi}=\bar{\phi}(\bar{b}^2, \bar{s})$ is a solution of (\ref{pde1***}) and (\ref{pde2***}) with $\bar{\kappa}=|\mu|$, then the following function
\begin{eqnarray}
\phi(b^2,s):=\f{\sqrt{|\mu|}\sqrt{\kappa-\mu b^2+\mu s^2}}{\kappa-\mu b^2}\bar\phi\left(\f{\mu b^2}{\kappa-\mu b^2}\frac{\mu}{\kappa},\; \f{|\mu| s}{\sqrt{\kappa-\mu b^2}\sqrt{\kappa-\mu b^2+\mu s^2}}\right)
\end{eqnarray}
satisfies (\ref{pde1**}) and (\ref{pde2**}).
\end{lem}
\begin{proof} Let  $\p=\phi(b^2,s)$ be a positive smooth function satisfying (\ref{pde1**}) and (\ref{pde2**}).
Take  $\alpha$ and $\beta$ satisfying  (\ref{pde1*****}) and define $\bar{\alpha}$ and $\bar{\beta}$ by (\ref{deformationab}). Then $\bar{\alpha}$ and $\bar{\beta}$ satisfy (\ref{pde1b*****}) by Lemma \ref{mupositive1}. For the  function $\bar{\phi} =\bar{\phi}(\bar{b}^2, \bar{s})$  in (\ref{barphi}), we have
\[
F=\alpha \phi ( b^2,s )=\bar{\alpha}\bar{\phi} (\bar{b}^2, \bar{s}),
\]
where $ s = \beta/\alpha$ and $\bar{s}=\bar{\beta}/\bar{\alpha}$.
By the same argument in \cite{yct-zhm-onan}, one can see that $\bar\p(\bar b^2,\bar s)$ satisfies (\ref{pde1***}) if and only if $\p(b^2,s)$ satisfies (\ref{pde1**}).
 Notice that $\bar{\alpha}$ and $\bar{\beta}$ satisfy (\ref{pde1b*****}) with $\bar{c}^2=|\mu |$, so $\bar\p=\bar\p(\bar b^2,\bar s)$ must satisfy (\ref{pde2***}) with $\bar{\kappa}= |\mu|$  by Theorem \ref{main1}.
 One can also verify this lemma by Maple program directly.

The converse is proved in a similar way.
\end{proof}

\section{Solutions}

In order to find Einstein general $(\alpha,\beta)$-metrics, we first take $\bar{\alpha}$ and $\bar{\beta}$ satisfying (\ref{pde1b*****}) with $\bar{c}^2 =|\mu|$, then we find function $\bar{\phi}=\bar{\phi}(\bar{b}^2, \bar{s})$ satisfying (\ref{pde1***}) and (\ref{pde2***}). Then by Theorem \ref{main1} the general $(\alpha,\beta)$-metric
$ F = \bar{\alpha}\bar{\phi}(\bar{b}^2, \bar{\beta}/\bar{\alpha})$  is Einstein with Ricci constant $K$. Using the above deformations, we  express $F$ in another form
$F=\alpha \phi (b^2, \beta/\alpha)$ where $\alpha$ and $\beta$ satisfy (\ref{pde1*****}) and
$\phi$ satisfies (\ref{pde1**}) and (\ref{pde2**}). By Theorem \ref{main1} again, we can see that $F $ is Einstein metric with Ricci constant $K$.

In this section, we shall solve (\ref{pde1***}) and (\ref{pde2***}). For simplicity, we shall remove the bars. We consider the following PDEs:
\begin{eqnarray}
&&\phi_{22}=2(\phi_1-s\phi_{12}) \label{pde1new}\\
&& \kappa [\psi^2 -(\psi_2+2s\psi_1)]=K\phi^2\label{pde2new}
\end{eqnarray}
where $\psi = (\phi_2+2s\phi_1)/(2\phi)$.

Making
\begin{eqnarray}\label{uv}
u=b^2-s^2,\qquad v=s.
\end{eqnarray}
(\ref{pde2new}) can be reexpressed as a simpler form
\begin{eqnarray}\label{pde5}
\kappa\left(\f{1}{\sqrt{\p}}\right)_{vv}-K\left(\f{1}{\sqrt{\p}}\right)^{-3}=0.
\end{eqnarray}

\begin{lem}
The non-constant solutions of Equation (\ref{pde5}) are given by
\begin{eqnarray}\label{sol01}
\p(u,v)=\frac{1}{p(u)\pm2\sqrt{-\sigma}v}
\end{eqnarray}
or
\begin{eqnarray}\label{sol02}
\p(u,v)=\frac{q(u)}{(p(u)+q(u)v)^2+\sigma},
\end{eqnarray}
where $\sigma=\frac{K}{\kappa}$.
\end{lem}
\begin{proof}(\ref{pde5}) can be written as
\begin{eqnarray}\label{eqn01}
\left(\frac{1}{\sqrt{\p}}\right)_{vv}=\sigma\left(\frac{1}{\sqrt{\p}}\right)^{-3}
\end{eqnarray}
Denote $w=\frac{1}{\sqrt{\phi}}$. Set $y=\frac{\ud w}{\ud v}$ and regard it as a function of $w$, then
$$\frac{\ud^2w}{\ud v^2}=\frac{\ud y}{\ud w}\cdot\frac{\ud w}{\ud v}=yy'.$$
Now equation (\ref{eqn01}) becomes
\begin{eqnarray}
yy'=\sigma w^{-3}.
\end{eqnarray}
The solutions of the above equation are given by
$$y^2=q(u)-\sigma w^{-2}$$
for some function $q(u)$, hence
$$\f{\ud w}{\ud v}=\pm\sqrt{q(u)-\sigma w^{-2}},$$
which means
$$\f{1}{2}\f{\ud w^2}{\ud v}=\pm\sqrt{q(u)w^2-\sigma}.$$
So $\p$ are given by (\ref{sol01}) when $q(u)=0$ or (\ref{sol02}) when $q(u)\neq0$.
\end{proof}

\begin{lem}
The non-constant solutions of Equation (\ref{pde1new}) and (\ref{pde2new}) are given by
\begin{eqnarray}\label{sol03}
\p(b^2,s)=\f{1}{2\sqrt{-\sigma}}\cdot\f{1}{\pm\sqrt{C-b^2+s^2}\pm s}
\end{eqnarray}
or
\begin{eqnarray}\label{sol05}
\p(b^2,s)=\f{q(u)}{q^2(u)(Dq(u)+v)^2+\sigma},
\end{eqnarray}
where $q\neq0$ is determined by the following equation
\begin{eqnarray}
D^2q^4+(u-C)q^2-\sigma=0.
\end{eqnarray}
$C$ and $D$ given above are both constant number.
\end{lem}
\begin{proof}
Firstly, under the change of variables as (\ref{uv}), (\ref{pde1new}) becomes
\begin{eqnarray}\label{eqn02}
\p_{vv}-2v\p_{uv}-4\p_u=0
\end{eqnarray}

(1) When $\p$ is given by (\ref{sol01}), then
\begin{eqnarray*}
\p_u=\f{-p'}{(p\pm2\sqrt{-\sigma}v)^2},\qquad\p_v=\pm\f{-2\sqrt{-\sigma}}{(p\pm2\sqrt{-\sigma}v)^2},\\
\p_{uv}=\pm\f{4\sqrt{-\sigma}p'}{(p\pm2\sqrt{-\sigma}v)^3},\qquad\p_{vv}=\f{-8\sigma}{(p\pm2\sqrt{-\sigma}v)^3}.
\end{eqnarray*}
So the equation (\ref{eqn02}) is equivalent to the following equation
\begin{eqnarray*}
2\sigma\pm2\sqrt{-\sigma}p'v-p'(p\pm2\sqrt{-\sigma}v)=0,
\end{eqnarray*}
and hence
\begin{eqnarray}\label{p1}
2\sigma-pp'=0.
\end{eqnarray}
When $\sigma=0$, $p$ thus $\p$ is a constant number in this case. When $\sigma\neq0$, we have
$$p=\pm2\sqrt{-\sigma}\cdot\sqrt{C-u}$$
for some constant number $C$. So $\p$ is given by (\ref{sol03}).

(2) When $\p$ is given by (\ref{sol02}), then
\begin{eqnarray*}
\p_u&=&\frac{q'}{[(p+qv)^2+\sigma]}-\frac{2(p+qv)q(p'+q'v)}{[(p+qv)^2+\sigma]^2},\\
\p_v&=&-\frac{2(p+qv)q^2}{[(p+qv)^2+\sigma]^2},\\
\p_{uv}&=&-\frac{4(p+qv)qq'}{[(p+qv)^2+\sigma]^2}-\frac{2q^2(p'+q'v)}{[(p+qv)^2+\sigma]^2}
+\frac{8(p+qv)^2q^2(p'+q'v)}{[(p+qv)^2+\sigma]^3},\\
\p_{vv}&=&-\frac{2 q^3}{[(p+qv)^2+\sigma]^2}+\frac{8(p+qv)^2q^3}{[(p+qv)^2+\sigma]^3}.
\end{eqnarray*}
So the equation (\ref{eqn02}) is equivalent to the following equation
\begin{eqnarray*}
A(u)v^3+B(u)v^2+C(u)v+D(u)=0,
\end{eqnarray*}
where
\begin{eqnarray*}
A(u)&=&2q^3(2pq'-qp')\\
B(u)&=&3q^2[(2(p^2+)q'+q^3)],\\
C(u)&=&6q[q(p^2+\sigma)p'+pq^3],\\
D(u)&=&(3p^2-\sigma)[2(p^2+\sigma)q'+q^3]-4p(p^2+\sigma)(2pq'-qp').
\end{eqnarray*}
It is easy to see that $A=B=C=D=0$, which means that $\p$ is a solution of (\ref{eqn02}), if and only if
\begin{eqnarray}
2pq'-qp'=0,\label{eqn03}\\
2(p^2+\sigma)q'+q^3=0\label{eqn04}.
\end{eqnarray}
(\ref{eqn03}) implies that
\begin{eqnarray}\label{pqD}
p=Dq^2
\end{eqnarray}
for some constant number $D$ or $q=0$, the later can be omitted since $\p=0$ in this case.

Plugging (\ref{pqD}) into (\ref{eqn04}) yields
$$\left(D^2q^4+\sigma\right)(q^2)'+q^4=0,$$
so $q=0$ or
\begin{eqnarray}\label{eqn05}
D^2q^4+(u-C)q^2-\sigma=0.
\end{eqnarray}
for some constant number $C$.
\end{proof}

\bigskip

The solutions of  (\ref{sol05}) can be reexpressed explictly as follows.\\

\begin{enumerate}
\item[(i)] $D=0$, $\phi$ is given by
\begin{eqnarray}\label{sol04}
\p(b^2,s)=\pm\frac{\sqrt{-\sigma^{-1}(C-b^2+s^2)}}{C-b^2}.
\end{eqnarray}
In this case, the corresponding general $\ab$-metrics are Riemannian metrics.\\
\item[(ii)] When $D\neq0$ and $\sigma=0$, $\phi$ is given  by
\begin{eqnarray}
\p(b^2,s)=\f{D}{\sqrt{C-b^2+s^2}(\sqrt{C-b^2+s^2}\pm s)^2},
\end{eqnarray}
In this case, the corresponding general $\ab$-metrics are Berwald's metrics.\\
\item[(iii)] When $D\neq0$ and $\sigma<0$, $\phi$ is given by
\begin{eqnarray}
\p(b^2,s)=\f{1}{2\sqrt{-\sigma}}\left(\f{1}{\pm\sqrt{C+2\sqrt{-\sigma}D-b^2+s^2}-s}
-\f{1}{\pm\sqrt{C-2\sqrt{-\sigma}D-b^2+s^2}-s}\right)
\end{eqnarray}
In this case, the corresponding general $\ab$-metrics are first given by \cite{szm-proj} as (39) in it.\\
\item[(iv)] When $D\neq0$, $\sigma>0$ and $q$ is real, $\phi$ is given by
\begin{eqnarray}
\p(b^2,s)=\f{1}{\sqrt{\sigma}}\Re\f{1}{\pm\sqrt{C+i2\sqrt{\sigma}D-b^2+s^2}+is}
\end{eqnarray}
In this case, the corresponding general $\ab$-metrics are Bryant's metrics.
\end{enumerate}


\noindent Zhongmin Shen\\
Department of Mathematical Sciences, Indiana University-Purdue University, Indianapolis, IN 46202-3216, USA\\
zshen@math.iupui.edu
\newline
\newline
\newline
\noindent Changtao Yu\\
School of Mathematical Sciences, South China Normal
University, Guangzhou, 510631, P.R. China\\
aizhenli@gmail.com
\end{document}